\documentclass[reqno]{amsart}
\usepackage{bbm}
\usepackage{amsfonts}
\usepackage{mathrsfs}
\usepackage{hyperref}
\usepackage{amssymb}
\usepackage{CJK}
 \newtheorem{thm}{Theorem}[section]
 \newtheorem{cor}[thm]{Corollary}
 \newtheorem{prop}[thm]{Corollary}

 \newtheorem{lem}[thm]{Lemma}

 \theoremstyle{remark}
 
 \numberwithin{equation}{section}
\DeclareMathOperator{\sign}{sign}

\newcommand{\order}{\textup{order}}
\newcommand{\f}{\mathbb{F}_q}
\newcommand{\ff}{\mathbb{F}_p}
\newcommand{\z}{\mathbb{Z}}

\begin{document}
\title[Counting polynomial subset sums]
  {  Counting polynomial subset sums}

\author{Jiyou Li}
\address{Department of Mathematics, Shanghai Jiao Tong University, Shanghai, P.R. China\\
Department of Mathematics, Massachusetts Institute of Technology, Cambridge, MA 02139-4307, USA}
\email{jiyouli@mit.edu}

\author{Daqing Wan}
\address{Department of Mathematics, University of California, Irvine, CA 92697-3875, USA}
\email{dwan@math.uci.edu}



\begin{abstract}

Let $D$ be a subset of a finite commutative ring $R$ with identity. Let $f(x)\in R[x]$ be a polynomial of positive degree $d$.
For integer $0\leq k \leq |D|$, we study the number $N_f(D,k,b)$ of
$k$-subsets $S\subseteq D$ such that
 \begin{align*}
 \sum_{x\in S} f(x)=b.
  \end{align*}
In this paper,  we establish several asymptotic formulas for $N_f(D,k, b)$, depending on the nature of the ring $R$ and $f$.

For $R=\z_n$, let $p=p(n)$ be the smallest prime divisor of $n$, $|D|=n-c \geq C_dn p^{-\frac 1d }+c$ and $f(x)=a_dx^d +\cdots +a_0\in \z[x]$ with $(a_d, \dots, a_1, n)=1$.
Then
 $$\left| N_f(D, k, b)-\frac{1}{n}{n-c \choose k}\right|\leq
 {\delta(n)(n-c)+(1-\delta(n))(C_dnp^{-\frac 1d}+c)+k-1\choose k},$$
partially answering an open question raised by Stanley \cite{St}, where $\delta(n)=\sum_{i\mid n, \mu(i)=-1}\frac 1 i$ and $C_d=e^{1.85d}$.  Furthermore, if $n$ is a prime power, then $\delta(n) =1/p$ and one can take $C_d=4.41$.

For $R=\f$ of characteristic $p$,  let $f(x)\in \f[x]$ be a polynomial of degree $d$ not divisible by $p$ and
 $D\subseteq \f$ with $|D|=q-c\geq (d-1)\sqrt{q}+c$. Then
$$\left| N_f(D, k, b)-\frac{1}{q}{q-c \choose k}\right|\leq
{\frac{q-c}{p}+\frac {p-1}{p}((d-1)q^{\frac 12}+c)+k-1 \choose k}.$$

If $f(x)=ax+b$, then this problem is precisely the well-known subset sum problem over a finite abelian
group. Let $G$ be a finite abelian group and let
$D\subseteq G$ with $|D|=|G|-c\geq c$.   Then
$$\left| N_x(D, k, b)-\frac{1}{|G|}{|G|-c \choose k}\right|\leq
{c + (|G|-2c)\delta(e(G))+k-1 \choose k},$$
where $e(G)$ is the exponent of $G$ and $\delta(n)=\sum_{i\mid n, \mu(i)=-1}\frac 1 i$.
In particular, we give a new short proof for the explicit counting formula for the case $D=G$.




\end{abstract}

\maketitle \numberwithin{equation}{section}
\newtheorem{theorem}{Theorem}[section]
\newtheorem{lemma}[theorem]{Lemma}
\newtheorem{example}[theorem]{Example}
\allowdisplaybreaks

\section{Introduction}
Let $D$ be a subset of a finite commutative ring $R$ with identity.
Let $f(x)\in R[x]$ be a polynomial of degree $d$.
Many problems from combinatorics and number theory are reduced to
computing the number $N_f(D,k,b)$, which is defined as the number of
$k$-subsets $S\subseteq D$ such that
 \begin{align*}
 \sum_{x\in S} f(x)=b.
  \end{align*}
  For example, when $R$ equals $\z_n$,  this problem was raised by Stanley \cite{St} (Page 136).

When $f(x)$ is linear,  we may just take $f(x)=x$.   The definition of $N(D,k,b):=N_x(D,k,b)$ is then defined for $R=G$ to be any finite abelian group (no ring structure is used).
The problem of computing $N(D,k,b)$ is then reduced to the counting version of the $k$-subset sum problem over $G$.

For $G=\z_n$, this problem is a well known \textbf{NP}-hard
problem in theoretical computer science. For a general
finite abelian group $G$, and an arbitrary $D\subseteq G$, determining if
$N(D, k, b)>0$ is an important difficult problem in algorithms and complexity. This has been studied extensively in recent
years,
especially over finite fields and over the group of rational points on an elliptic curve over a finite field, because of
their important applications in coding theory and cryptography, see \cite{Ch},  \cite{ZFW}, \cite{ZW} and the references there.
One expects that the problem is easier if $|D|$ is large
compared to $|G|$ or $D$ has some algebraic structure. For example, the dynamic programming algorithm gives a polynomial time algorithm
to compute $N(D, k, b)$ if $|D| > |G|^{\epsilon}$ for some positive constant $\epsilon>0$.
In the extreme case that $D=G$, an explicit formula for $N(D,
k, b)$  was obtained in Li and Wan \cite{LW3}, see also Kosters \cite{Kos} for a new proof and an improvement.
The sieving argument in \cite{LW3} has been used to obtain a good asymptotic formula for
$N(D,k, b)$ in the more general case that $D$ is close to $G$, for instance, when $|D|\geq \frac 2 3
|G|$. In the case that $G$ is the group of rational points on an elliptical curve over a finite field, please refer to
\cite{LWZ} for a concrete example.


For $G=\mathbb{Z}_n$, the finite cyclic group of $n$ elements and $D=G=\mathbb{Z}_n$, an old result of Ramanathan (1945) gives an explicit formula for
$N({\mathbb{Z}_n}, k,b)$  by using equalities involving Ramanujan's trigonometric sums. A formula for $\sum_{k}N(D, k,b)$ and several
generalizations were given by Stanley and Yoder \cite{SY}, Kitchloo and Patcher \cite{KP}.

 When $G=\mathbb{Z}_p$ is the finite cyclic group of prime order $p$ and $|D| \gg p^{2/3}$ is arbitrary, Erd\H{o}s and Heilbronn proved in their famous paper \cite{EH} that $\sum_{k}N(D, k,b)=\frac {2^p}{p}(1+o(1))$ when $p$ tends to infinity.

When $G$ is the additive group of a finite field ${\mathbb F}_q$ and $|G|-|D|$ is bounded by a constant, an explicit
formula for $N(D, k,b)$ was given in \cite{LW1}.
When $G$ is an arbitrary finite abelian group, and $D=G$ or $D=G^*$, an explicit and efficiently computable formula for $N(D, k,b)$ was given in \cite{LW3}. Kosters \cite{K} gave a different and shorter proof by using methods of group rings.
In this paper, we will give a third short proof.

We  also obtain a general bound for the $k$-subset sum
problem over $G$, which significantly generalizes previous results which assumed $D$
to be very close to $G$. This will be explained shortly later for the case $|G|=p$.

\begin{thm}\label{theorem1.3}Let $G$ be a finite abelian group of order $|G|$.
 Let $D\subseteq G$ with $|D|=|G|-c\geq c$.  Let $N(D, k, b)$
be the number of $k$-subsets in $D$ which sums to $b$. Then
$$\left| N(D, k, b)-\frac{1}{|G|}{|G|-c \choose k}\right|\leq
{c + (|G|-2c)(\sum_{i\mid e(G), \mu(i)=-1}\frac {1}{i})+k-1 \choose k},$$
 where $e(G)$ is the exponent of $G$, which is defined as the maximal order of a nonzero element in $G$.
\end{thm}

In order for this bound to be non-trivial, at least $k$ and $c$ need to satisfy
$$|G|-c>(|G|-2c)(\sum_{i\mid e(G), \mu(i)=-1}\frac {1}{i}) +k+ c.$$

\begin{prop}  Let $G$ be a finite elementary abelian $p$-group (thus $e(G)=p$). Then,
$$\left| N(D, k, b)-\frac{1}{|G|}{|G|-c \choose k}\right|\leq
{\frac{(|G|-2c)}{p}+c+k-1 \choose k}.$$
\end{prop}

In the case $|G|=p$, to obtain a non-trivial estimate, one needs to solve
$$p-c >\frac{(p-2c)}{p}+k+c.$$
Asymptotically, for smaller $k$, we could take $c$ as large as $p/2$.

Let us turn to the cases for general $f(x)$.
Few results are known for the number $N_f(D, k, b)$ when $f(x)$ is a polynomial of higher degree. In the case that
$f(x)$ is the simplest monomial $x^d$, $R$ is the prime field $\ff$ and
$D=\ff^*$, it was first proved by Odlyzko-Stanley \cite{OS} that
 \begin{align*}\left |N_{x^d}(\ff^*, b)-\frac{2^{p-1}}{p}\right|\leq
e^{O(d\sqrt{p}\log{p})},
\end{align*}
where $N_{x^d}(\ff^*, b)=\sum_{k=0}^{p-1} N_{x^d}(\ff^*, k,  b)$.

For a general finite field $R={\mathbb{F}}_q$, the
finite field of $q=p^t$ elements, Zhu and Wan \cite{W} proved the following more precise result:
\begin{align*}\left | N_{x^d}(\f^*, k, b)- \frac{1}{q}{q-1 \choose k}  \right|\leq
2q^{-1/2}{ d\sqrt{q}+q/p+k  \choose k}.
\end{align*}
Since $N_{x^d}(\f^*, b)=\sum_{k=0}^{q-1} N_{x^d}(\f^*, k, b)$, one can then deduce
the following explicit bound
\begin{align*}\left |N_{x^d}(\f^*, b)-\frac{2^{q-1}}{q}\right|\leq
\frac {4p}{\sqrt{2\pi} q}e^{(d\sqrt{q}+q/p)\log{q}},
 \end{align*}
which extends the Odlyzko-Stanley bound from a prime finite field to a general finite field.
Note that simply replacing $p$ with $q$ in the Odlyzko-Stanley bound is not known to be true
and is probably not true if $q$ is a high power of $p$. It is true if $q=p^2$.

These bounds are nontrivial only for $d\leq \sqrt{q}$. When $q=p$
is prime, a series of subsequent work had been made by Garcia-Voloch, Shparlinski, Heath-Brown, Heath-Brown-
Konyagin and Konyagin. They used variations of Stepanov¡¯s method and released the limit on the degree to $d\leq {p}^{3/4-\epsilon}$. For more details, please refer to \cite{B2}.
Using their remarkable Gauss sum bound proved by using additive combinatorics and harmonic analysis,  Bourgain, Glibichuk and Konyagin \cite{BGK, BK} proved that if $d < p^{1-\delta}$ for some
constant $\delta >0$, then there is a constant $0<\epsilon=\epsilon(\delta)<\delta$ such that
\begin{align*}\left |N_{x^d}(\ff^*,
b)-\frac{2^{p-1}}{p}\right|\leq e^{O(p^{1-\epsilon})}.
 \end{align*}

By combining Bourgain's bound and Li and Wan's  sieving technique
\cite{LW2}, Li \cite{Li} proved a refined result that if $d<p^{1-\delta}$, then there
is a constant $0<\epsilon=\epsilon(\delta)<\delta$ such that
\begin{align*}
\left| N_{x^d}(\ff^*, k, b)-\frac{1}{p}{p-1 \choose k}  \right|\leq
{p^{1-\epsilon}+dk-d \choose k}.
 \end{align*}
It would be interesting to extend this type of result to a general finite field of characteristic $p$.

In this paper, we obtain several asymptotic formulas for $N_f(D, k, b)$  when $f(x)$ is a
general higher degree polynomial. In the case that $R=\z_n$,  the finite ring of $n$ residues mod $n$, and $f$ is a polynomial of degree $d$ over the integers,  we have the following bound, proved using Hua's
bound for exponential sums and our sieving technique.

  \begin{thm}\label{theorem1.2} Let $R=\z_n$ and  $p=p(n)$ be the smallest prime divisor of $n$.
Assume $|D|=n-c \geq C_dn p^{-\frac 1d }+c$ and $f(x)=a_dx^d +\cdots +a_0\in \z[x]$ with $(a_d, \dots, a_1, n)=1$.
Then we have
 $$\left| N_f(D, k, b)-\frac{1}{n}{n-c \choose k}\right|\leq
{\delta(n)(n-c)+(1-\delta(n))(C_dnp^{-\frac 1d}+c)+k-1\choose k},$$
 where $\delta(n)=\sum_{i\mid n, \mu(i)=-1}\frac 1 i$ and  $C_d=e^{1.85d}$.
  Furthermore, if $n$ is a prime power, then $\delta(n) =1/p$ and the constant
$e^{1.85d}$ can be improved to the absolute constant $4.41$.
 \end{thm}

Note that the above bound is pretty good for $n$ with only large prime factors so that
$\delta(n)=\sum_{i\mid n, \mu(i)=-1}\frac {1}{i}$ is relatively small.

When $R=\f$ and $f$ is a polynomial of degree $d$ over
$\f$,  we obtain a better bound thanks to the Weil bound. In this
case, for simplicity, we suppose  that $f(x)\in \f[x]$ is a polynomial of degree
$d$, $d$ is not divisible by $p$ and $d<q$ since $x^q =x$ for all $x\in \f$.

 \begin{thm}\label{theorem1.1} Let $f(x)\in \f[x]$ be a polynomial of degree $d$ not divisible by $p$.
 For $R=\f$ and $|D|=q-c\geq (d-1)\sqrt{q}+c$, we have
$$\left| N_f(D, k, b)-\frac{1}{q}{q-c \choose k}\right|\leq
{\frac{q-c}{p}+\frac {p-1}{p}((d-1)q^{\frac 12}+c)+k-1 \choose k}.$$
   \end{thm}

In particular, if $q=p$ is a prime,  then we have a nice ``quadratic
root" bound.

 \begin{cor} Let $f(x)\in \ff[x]$ be a polynomial of degree $0<d<p$.
  For $R=\ff$ and $|D|=p-c\geq (d-1)\sqrt{p}+c$, we have
$$\left| N_f(D, k, b)-\frac{1}{p}{p-c \choose k}\right|\leq
{(d-1)p^{\frac 12}+c+k \choose k}.$$
   \end{cor}

The paper is organized as follows. In Section 2,  we briefly review a distinct coordinate sieving formula. In Section 3,
we establish a general formula for general ring $R$. In the remaining sections, several
more explicit formula are derived.

{\bf Notations}. For $x\in\mathbb{R}$, let  $(x)_0=1$
 and $(x)_k=x (x-1) \cdots (x-k+1)$
for $k\in $ $\mathbb{Z^+}$. For $k\in \mathbb{N}$, ${x \choose k}$
is the binomial coefficient defined by ${x \choose k}=\frac
{(x)_k}{k!}$. For a power series $f(x)$, $[x^k]f(x)$ denotes the coefficient of $x^k$ in $f(x)$.
$\lfloor x \rfloor$ always denotes the largest integer not greater than $x$.

\section{A distinct coordinate sieving formula}
For the purpose of our proof, we briefly introduce the  sieving formula
discovered by Li and Wan \cite{LW2}.
Roughly speaking, this formula significantly improves the classical
inclusion-exclusion sieve for distinct coordinate counting
problems. We cite it here without proof. The first proof of this formula was given in \cite{LW2}.
For a different proof by the theory of partial order please refer to
\cite{LW3}.

Let $\Omega$ be a finite set, and let $\Omega^k$ be the Cartesian
product of $k$ copies of $\Omega$. Let $X$ be a
subset of $\Omega^k$. 
Define $\overline{X}=\{(x_1,x_2,\cdots,x_k)\in X \ | \ x_i\ne x_j,
\forall i\ne j\}.$
Denote $S_k$ to be the symmetric group on $n$ elements.
 For a permutation $\tau$ in $S_k$, the sign of $\tau$ is defined by
  $\sign(\tau)=(-1)^{k-l(\tau)}$, where $l(\tau)$ is the number of
 cycles of $\tau$ including the trivial ones.
Suppose we have the factorization $\tau=(i_1i_2\cdots i_{a_1})
  (j_1j_2\cdots j_{a_2})\cdots(l_1l_2\cdots l_{a_s})$
  with $1\leq a_i, 1 \leq i\leq s$, then define
  \hskip 1.0cm
  \begin{align} \label{1.1}
     X_{\tau}=\left\{
(x_1,\dots,x_k)\in X,
 x_{i_1}=\cdots=x_{i_{a_1}},\cdots, x_{l_1}=\cdots=x_{l_{a_s}}
 \right\}.
\end{align}
  Now we will state our sieve
formula.   We notice that there are many other interesting
corollaries of this formula \cite{LW2, LW3}.

\begin{thm} \label{thm1.0}
Let $f(x_1,x_2,\dots,x_k)$ be any complex valued function defined over
$X$. Then
  \begin{align*}
 \sum_{x \in \overline{X}}f(x_1,x_2,\dots,x_k)=\sum_{\tau\in S_k}{\sign(\tau)\sum_{x \in
X_{\tau} } f(x_1,x_2,\dots,x_k)}.
    \end{align*}
 \end{thm}


Note that in many situations the sum  $\sum_{x \in
X_{\tau} } f(x_1,x_2,\dots,x_k)$ is much easier to compute compared to the left one.

 $S_k$ acts on $\Omega^k$ naturally by
permuting the coordinates. That is, for $\tau\in S_k$ and
$x=(x_1,x_2,\dots,x_k)\in \Omega^k$, $\tau\circ
x=(x_{\tau(1)},x_{\tau(2)},\dots,x_{\tau(k)}).$
  A subset $X$ in $\Omega^k$ is said to be symmetric if for any $x\in X$ and
any $\tau\in S_k$, $\tau\circ x \in X $.
%
%
%
 For $\tau\in S_k$, denote by $\overline{\tau}$
 the conjugacy class represented by $\tau$ and sometimes it is more convenient to view it
 as the set of permutations conjugate to $\tau$.
Conversely, for a conjugacy class $\overline{\tau}\in C_k$, just
let $\tau$ denote a representative permutation in this class. 

In particular, since two permutations in $S_k$ are conjugate if and only if they
have the same type of cycle structure,  if  $X$ is symmetric and $f$ is a symmetric function
under the action of $S_k$,  then we have the following simpler formula.
\begin{prop} \label{thm1.1} Let $C_k$ be the set of conjugacy  classes
 of $S_k$.  If $X$ is symmetric and $f$ is symmetric, then
 \begin{align}\label{7} \sum_{x \in \overline{X}}f(x_1,x_2,\dots,x_k)=\sum_{\overline{\tau} \in C_k}\sign(\tau) C(\tau)
 \sum_{x \in X_{\tau} } f(x_1,x_2,\dots, x_k),
  \end{align} where $C(\tau)$ is the number of permutations conjugate to
  $\tau$.
\end{prop}

%

%
%
%
%
%

\begin{lem}\label{lem2.4} We have the following inequality for the coefficients of rational functions. For positive integers $m$ and $n$,
$$[x^k]\frac {1} { (1-x^m)^n}\leq [x^k] \frac {1} { (1-x)^n}.$$
\end{lem}

\begin{proof}
Since  $$\frac 1 {(1-x)^n}=\sum_{k=0}^\infty {k+n-1 \choose k} x^k,$$ we have
$$[x^k]\frac 1 {(1-x^m)^n}\leq {[k/m]+n-1 \choose [k/m]} \leq {k+n-1 \choose k}=[x^k]\frac 1 {(1-x)^n}.$$
\end{proof}

\begin{lem}\label{lem2.5}
If for all integers $k\geq 0$,  we have
$$[x^k]f_1(x)\leq[x^k]g_1(x), \ [x^k]f_2(x)\leq[x^k]g_2(x),$$
then for all integers $k\geq 0$,
$$[x^k]f_1(x)f_2(x)\leq[x^k]g_1(x)g_2(x). $$
\end{lem}

\begin{proof}
$$[x^k]f_1(x)f_2(x)=\sum_{i=0}^k[x^i] f_1(x)\cdot[x^{k-i}]f_2(x)\leq \sum_{i=0}^k[x^i]g_1(x)[x^{k-i}]g_2(x)=
[x^k]g_1(x)g_2(x).$$
\end{proof}

%
%

\begin{lem}\label{lem2.3} If $a, b$ are integers and $0\leq b\leq a$, then we have the inequality on the coefficients for rational functions.
$$[x^k]\frac {(1-x^{pq})^b}  { (1-x^p)^{a} }\leq [x^k] \frac {1} { (1-x^p)^{a}}.$$
\end{lem}

\begin{proof}
Applying Lemma \ref{lem2.4}  and Lemma \ref{lem2.5}, we have
 \begin{align*}
[x^k]\frac {(1-x^{pq})^b}{ (1-x^p)^{a} }&=[x^k] (\frac {1-x^{pq}}{1-x^p})^b \cdot \frac {1}{ (1-x^p)^{a-b} }\\
&=[x^k] (1+x^p+\cdots+x^{(q-1)p})^b\cdot  \frac {1} { (1-x^p)^{a-b}   }\\
&\leq[x^k] (1+x^p+\cdots+x^{(q-1)p}+\cdots)^b\cdot \frac {1} { (1-x^p)^{a-b}   }\\
&= [x^k] \frac {1} { (1-x^p)^a}.
\end{align*}
\end{proof}

We now establish a combinatorial upper bound  which is crucial for the proof of our main results.
  A permutation $\tau\in S_k$ is said to be of type
$(c_1,c_2,\cdots,c_k)$ if $\tau$ has exactly $c_i$ cycles of length
$i$.  Note that $\sum_{i=1}^k ic_i=k$. Let $N(c_1,c_2,\dots,c_k)$ be
the number of permutations in $S_k$ of type $(c_1,c_2,\dots,c_k)$.
It is well known  that
$$N(c_1,c_2,\dots,c_k)=\frac {k!} {1^{c_1}c_1! 2^{c_2}c_2!\cdots k^{c_k}c_k!},$$
and we then define the generating function
\begin{align*}C_k(t_1,t_2,\cdots,t_k)= \sum_{\sum
ic_i=k} N(c_1,c_2,\cdots,c_k)t_1^{c_1}t_2^{c_2}\cdots t_k^{c_k}.
 \end{align*}

\begin{lem} \label{lem6.3}
Let $q\geq s$ be two positive real numbers.
 If $t_i=q$ for $(i, d)>1$ and $t_i=s$ for $(i, d)=1$, then we have the bound
\begin{align*}
C_k(\overbrace{s,\cdots,s}^{(i,d)=1},q,\overbrace{s,\cdots,s}^{(i,d)=1},q,
\cdots) &=\sum_{\sum
ic_i=k} N(c_1,c_2,\cdots,c_k)s^{c_1}s^{c_2}\cdots q^{c_d}s^{c_{d+1}}\cdots \\
&\leq ( s+(q-s)(\sum_{i\mid d, \mu(i)=-1}\frac {1}{i})+k-1)_k.
 \end{align*}
\end{lem}

\begin{proof}
Suppose $d$ has the  prime factorization $d=\prod_{j=1}^t p_j^{s_j}$.
By the definition of the exponential generating function, we have
$$\sum_{k\geq 0}C_k(t_1,t_2,\cdots,t_k)\frac {u^k}{k!}=e^{ut_1+u^2\cdot\frac{t_2}2+
u^3\cdot\frac {t_3}3+\cdots}.$$
By the conditions $t_i=q$ for $(i, d)>1$ and $t_i=s$ for $(i, d)=1$, we deduce
\begin{eqnarray*} C_k(\overbrace{s,\cdots,s}^{(i, d)=1},q,\overbrace{s,\cdots,s}^{(i, d)=1},q,
\cdots)&=&\left[\frac {u^k}{k!}\right]e^{us+u^2\cdot\frac s
2+\cdots+u^{d-1}\cdot\frac {s} {d-1}+u^d\cdot\frac {q} d+u^{d+1}\cdot\frac {s}{d+1} \cdots}\\
&=&\left[\frac {u^k}{k!}\right]e^{s\sum_{i}\frac {u^i}{i}+(q-s)\sum_{(i, d)>1}\frac {u^i}{i}}.
\end{eqnarray*}
Using the inclusion-exclusion, the above expression can be re-written as

\begin{eqnarray*}
 &&\left[\frac {u^k}{k!}\right]e^{s\sum_{i}\frac {u^i}{i}+(q-s)\left(\sum_{p_1\mid i}\frac {u^i}{i}+\sum_{p_2\mid i}\frac
 {u^i}{i}+ \cdots-\sum_{p_1p_2\mid i}\frac {u^i}{i}- \cdots\right)}\\
&=&\left[\frac {u^k}{k!}\right]e^{-s\log{\left(1-u\right)}-\frac{q-s}{p_1}\log(1-u^{p_1})-
\frac{q-s}{p_2}\log(1-u^{p_2})-\cdots+\frac{q-s}{p_1p_2}\log(1-u^{p_1p_2})+\cdots}\\
&=&\left[\frac {u^k}{k!}\right]\frac {(1-u^{p_1p_2})^{\frac {q-s} {p_1p_2}}\cdots} {(1-u)^{s}(1-u^{p_1})^{\frac {q-s}
{p_1}}(1-u^{p_2})^{\frac {q-s} {p_2}}\cdots}\\
&\leq& \left[\frac {u^k}{k!}\right]\frac {1} {(1-u)^s (1-u^{p_1})^{\frac {q-s}
{p_1}}(1-u^{p_2})^{\frac {q-s} {p_2}}\cdots (1-u^{p_1p_2p_3})^{\frac {q-s}
{p_1p_2p_3}}\cdots}\\
&\leq& \left[\frac {u^k}{k!}\right]\frac {1} {(1-u)^s (1-u)^{\frac {q-s}
{p_1}}(1-u)^{\frac {q-s} {p_2}}\cdots (1-u)^{\frac {q-s}
{p_1p_2p_3}}\cdots}\\
&=&{k!} {  s +(q-s)(\sum_{i\mid d, \mu(i)=-1}\frac {1}{i})+k-1\choose k}\\
&=&( s+(q-s)(\sum_{i\mid d, \mu(i)=-1}\frac {1}{i})+k-1)_k.
\end{eqnarray*}
In the  above inequality step, we used Lemma \ref{lem2.3} and Lemma \ref{lem2.4}.
\end{proof}

In the same spirit, a simpler special case is
the following lemma and the proof is omitted.

\begin{lem} \label{lem6.4}
Let $q\geq s$ be two non negative real numbers.
 If $t_i=q$ for $ d\mid i$ and $t_i=s$ for $d\nmid i$, then we have
\begin{align*}
C_k(\overbrace{s,\cdots,s}^{d \nmid i },q,\overbrace{s,\cdots,s}^{d \nmid i},q,
\cdots)&=\sum_{\sum
ic_i=k} N(c_1,c_2,\cdots,c_k)s^{c_1}s^{c_2}\cdots q^{c_d}s^{c_{d+1}}\cdots \\
&=\left[\frac {u^k}{k!}\right]\frac {1} {(1-u)^s (1-u^d)^{\frac {q-s} {d}}}.\\
&\leq\left[\frac {u^k}{k!}\right]\frac {1} {(1-u)^s (1-u)^{\frac {q-s} {d}}}\\
&=(s+(q-s)/d+k-1)_k.
  \end{align*}
 \end{lem}

\section{General Case $R$}

 Let $\psi$ denote an additive character from $G=(R,+)$,
 the additive group of $R$, to the group of all nonzero complex numbers $\mathbb{C}^*$.
Let  $\psi_0$ be the
principal character sending each element in $G$ to 1. Denote by
$\hat{G}$ the group of additive characters of $G$, which is isomorphic to $G$.

\begin{lem}\label{lem3.1}  Suppose that $|R|=q$ and $D\subseteq R$ with $|D|=m$. For a fixed
polynomial $f(x)\in R[x]$, let $N_f(D, k, b)$  be the number of
$k$-subsets $S\subseteq D$ such that $\sum_{x\in S}f(x)=b$. Then
  \begin{align*}
 k!N_f(D, k, b)&=\frac{1}{q} {(m)_k}+\frac{1}{q} \sum_{\psi\ne \psi_0}
 \psi^{-1}(b) \sum_{\tau\in
C_{k}}\sign(\tau)C(\tau) F_{\tau}(\psi),
  \end{align*}
  where  $C_{k}$ is the set of all conjugacy classes
 of $S_{k}$ and $C(\tau)$ counts the number of permutations conjugate to
 $\tau$, and
   \begin{align*}
F_{\tau}(\psi)=\prod_{i=1}^{k}(\sum_{a\in D}\psi^i(f(a)))^{c_i}.
\end{align*}

\end{lem}

 \begin{proof}
Let $X=D\times D \times \cdots \times D$ be the Cartesian product of
$k$ copies of $D$. Define $  \overline{X} =\left\{
(x_1,x_2,\dots,x_{k} )\in D^k \mid
 x_i\not=x_j,~ \forall i\ne j\} \right\}$ to be the set of all distinct configurations
 in $X$. It is clear that $|X|=m^k$ and
$|\overline{X}|=(m)_k$. Applying the orthogonal relations of the
characters, one deduces that
\begin{align*}
k!N_f(D, k, b)&=\frac{1}{q} \sum_{(x_1, x_2,\dots x_k) \in \overline{X}}
\sum_{\psi\in \hat{G}}\psi(f(x_1)+f(x_2)+\cdots +f(x_k)-b)\\
&=\frac{1}{q}  {(m)_k}+\frac{1}{q} \sum_{\psi\ne
\psi_0}\psi^{-1}(b)\sum_{(x_1,x_2,\dots x_k)
\in\overline{X}}\prod_{i=1}^{k} \psi(f(x_i)).
\end{align*}
For $\psi\ne \psi_0$, let $f_{\psi}(x)=
f_{\psi}(x_1,x_2,\dots,x_{k})= \prod_{i=1}^{k}\psi(f(x_i))$.  For
  $\tau\in S_k$,  let
$$F_{\tau}(\psi)=\sum_{x\in X_{\tau}}f_{\psi}(x)=\sum_{x \in X_{\tau}}\prod_{i=1}^{k} \psi(f(x_i)),$$
where $X_{\tau}$ is defined as in equation (\ref{1.1}). Obviously $X$ is
symmetric and $f_{\psi}(x_1,x_2,\dots,x_{k})$ is also symmetric on $X$.
Applying equation (\ref{7}) in Corollary \ref{thm1.1}, we have
  \begin{align*}
 k!N_f(D, k, b)&=\frac{1}{q} {(m)_k}+\frac{1}{q} \sum_{\psi\ne \psi_0}
 \psi^{-1}(b) \sum_{\tau\in
C_{k}}\sign(\tau)C(\tau) F_{\tau}(\psi),
  \end{align*}
 where $C_{k}$ is the set of all conjugacy classes
 of $S_{k}$ and $C(\tau)$ counts the number of permutations conjugate to $\tau$.
For  $\tau\in C_{k}$, assume $\tau$ is of type
$(c_1,c_2,\dots,c_{k})$, where $c_i$ is the number of $i$-cycles in
$\tau$ for $1 \leq i\leq k$. Note that $\sum_{i=1}^{k} ic_i=k$.
Write
$$\tau=(i_1)(i_2)\cdots(i_{c_1})(i_{c_1+1}i_{c_1+2})(i_{c_1+3}i_{c_1+4})\cdots
(i_{c_1+2c_2-1}i_{c_1+2c_2})\cdots.$$
One checks that
  \begin{align*}
      X_{\tau}=\left\{
(x_1,\dots,x_k)\in D^k,
 x_{i_{c_1+1}}=x_{i_{c_1+2}}, \cdots,  x_{i_{c_1+2c_2-1}}=x_{i_{c_1+2c_2}},\cdots
 \right\}.
\end{align*}
Then we have
  \begin{align*}
F_{\tau}(\psi)&=\sum_{x \in X_{\tau}}\prod_{i=1}^{k} \psi(f(x_i))\\
&=\sum_{x \in X_{\tau}}\prod_{i=1}^{c_1}
\psi(f(x_i))\prod_{i=1}^{c_2}
\psi^2(f(x_{c_1+2i}))\cdots\prod_{i=1}^{c_k} \psi^k(f(x_{c_1+c_2+\cdots+k i}))\\
 &=\prod_{i=1}^{k}(\sum_{a\in D}\psi^i(f(a)))^{c_i}.
\end{align*}
\end{proof}

The above lemma reduces the study of the asymptotic formula for $N_f(D, k,b)$ to the estimate of the
partial character sum $\sum_{a\in D} \psi(f(a))$ and another sum through $\psi$.
 This is very difficult in general. However, if either $D$ is large compared to
$R$, or $D$ and $f(x)$ have some nice algebraic structures, one expects non-trivial estimates.
 One important example is the case that $D=\mathbb{F}_p^*$ and  $f(x)=x^d$. As we have mentioned in the introduction section, a series of works by Garcia-Voloch, Heath Brown, Konyagin-Shparlinski, Konyagin using variants of
Stepanov¡¯s method ($d <p^{3/4-\epsilon}$), and by Bourgain and Konyagin using additive combinatorics and harmonic analysis ($ d <p^{1-\epsilon}$)) shows that in this case $D$ has a nice pseudo random property.

 We are now ready to use the above lemma to prove our main results
by estimating various partial character sums and different summations in
different cases.

\section{The Residue Ring Case $R=\z_n$}

We first recall the following results on character sums over the residue class ring.

\begin{lem}[Hua and Lu \cite{Hua1, Lu}]\label{5.1}
Suppose $\psi$ is a primitive additive character of the
group $\z_n$.
  Let $f(x)=\sum_{i=0}^d a_i x^i \in \mathbb{Z}[x]$ be a polynomial of positive degree $d$.
  If $(a_1, \cdots, a_d, n)=1$, then
\begin{align*}
 |\sum_{x \in \z_n}\psi(f(x))|\leq e^{1.85d}n^{1-\frac 1d}.
\end{align*}
Thus if $D\subseteq \z_n$ with $|D|=n-c$, then
\begin{align}\label{4.1}
 |\sum_{x \in D}\psi(f(x))|\leq e^{1.85d}n^{1-\frac 1d}+c.
\end{align}
\end{lem}

For $d\geq 3$, the bound (5.1) can be improved to
\begin{align*}
 |\sum_{x \in \z_n}\psi(f(x))|\leq e^{1.74d}n^{1-\frac 1d}.
\end{align*}
by Ding and Qi \cite{DQ}. See also Ste\v{c}kin \cite{Ste} for
an asymptotically  better but not explicit bound for large $d$.

When $n$ is a prime power, Hua \cite{Hua1, Hua2, Hua3} first obtained the bound
\begin{align*}
 |\sum_{x \in \z_n}\psi(f(x))|\leq d^3 n^{1-\frac 1d},
\end{align*} and it was improved by many mathematicians including Chen, Chalk, Ding, Loh, Lu, Mit'kin, Ne\v{c}aev and Ste\v{c}kin \cite{CZ}.
The current best bound is proved by Cochrane and Zheng.
\begin{lem}[Cochrane and Zheng, \cite{CZ0}]
Suppose $\psi$ is a primitive additive character of the
group $\z_n$.
 Let $f(x)=\sum_{i=0}^d a_i x^i \in \mathbb{Z}[x]$ be a polynomial of positive degree $d$.
 Assume $n=p^t$ and $(a_1, \cdots, a_d, p)=1$.
Then
\begin{align*}
 |\sum_{x \in \z_n}\psi(f(x))|\leq 4.41 n^{1-\frac 1d}.
\end{align*}
Similarly, if $D\subseteq \z_n$ with $|D|=n-c$, then
\begin{align*}
 |\sum_{x \in D}\psi(f(x))|\leq 4.41 n^{1-\frac 1d}+c.
\end{align*}
\end{lem}

For readers interested in the exponential sums over $\z_n$, we refer
to a good survey by Cochrane and Zheng \cite{CZ}.


{\bf Proof of Theorem for $R=\z_n$}.
Let $\psi_0$ be the principal character sending each element in
$\z_n$ to 1. Also denote by $\hat{\z}_n$ the group of additive
characters of $\z_n$. Let $N_f(D,
k, b)$ be the number of $k$-subsets $S\subseteq D$ such that
$\sum_{x\in S}f(x)=b$.  Write $|D|=m$. Applying Lemma
\ref{lem3.1},  we have

 \begin{align*}
 k!N_f(D, k, b)&=\frac{1}{n} {(m)_k}+\frac{1}{n} \sum_{\psi\in \hat{\z}_n, \psi\ne \psi_0}
 \psi^{-1}(b) \sum_{\tau\in
C_{k}}\sign(\tau)C(\tau) F_{\tau}(\psi),
  \end{align*}
  where  $C_{k}$ is the set of all conjugacy classes
 of $S_{k}$ and $C(\tau)$ counts the number of permutations conjugate to
 $\tau$, and
   \begin{align*}
F_{\tau}(\psi)=\prod_{i=1}^{k}(\sum_{a\in D}\psi^i(f(a)))^{c_i}.
\end{align*}
Let $C_d = e^{1.85d}$ for general $n$ and $C_d=4.41$ for prime power $n=p^t$.
Applying equation \ref{4.1} in Lemma \ref{5.1}, if  $\psi$ is  primitive, then
$$|F_{\tau}(\psi)| \leq
m^{\sum_{i=1}^{k}c_im_i(\psi)}(C_dn^{1-\frac 1d
}+c)^{\sum_{i=1}^{k} c_i(1-m_i(\psi))},$$  where $m_i(\psi)$ is
defined as follows:
 $ m_i(\psi)=1$ if $(i,n)>1$ and  $ m_i(\psi)=0$ if $(i,n)=1$.
 Similarly, if $\order(\psi)=h, h\mid n$, then
 \begin{align*}
 |\sum_{x \in \z_n}\psi(f(x))|=\frac n h |\sum_{x \in \z_h}\psi(f(x))|\leq C_d nh^{-\frac 1d}.
\end{align*}
Thus
 $$|F_{\tau}(\psi)| \leq
m^{\sum_{i=1}^{k}c_im_i(\psi)}(C_dn h^{-\frac 1d
}+c)^{\sum_{i=1}^{k} c_i(1-m_i(\psi))},$$
where  $ m_i(\psi)=1$ if $(i,h)>1$ and  $ m_i(\psi)=0$ if $(i,h)=1$.

Let $p=p(n)$ be the smallest prime divisor of $n$.
Assume
$$m \geq \max_{h\mid n, h\ne 1} \{C_dn h^{-\frac 1d
}+c\}=C_dn p^{-\frac 1d
}+c.$$
Then we have
   \begin{align*}
 k!N_f(D, k, b)&=\frac{1}{n} {(m)_k}+\frac{1}{n} \sum_{1\ne h\mid n}\sum_{\psi, \order(\psi)=h}
 \psi^{-1}(b) \sum_{\tau\in
C_{k}}\sign(\tau)C(\tau) F_{\tau}(\psi)\\
&\geq \frac{1}{n} {(m)_k}-\frac{1}{n}
 \sum_{1\ne h\mid n}\phi(h)\sum_{\tau\in
C_{k}}C(\tau) m^{\sum_{j=1, (h, j)>1}^kc_{j}}( C_dnh^{-\frac
1d }+c)^{\sum_{j=1, (h, j)=1}^{k}
c_j}\\
 &\geq\! \frac{1}{n} {(m)_k}- \frac{1}{n}
 \sum_{\!1\ne h\mid n}\phi(t)( C_dnh^{-\frac 1d}+c+\! (m- C_dnh^{-\frac 1d}-c)(\!\!\sum_{i\mid h, \mu(i)=-1}\frac {1}{i})+k-1)_k
 \end{align*}

 Define $\delta(h)=\sum_{i\mid h, \mu(i)=-1}\frac 1 i $. Obviously $\max\{\delta(h), h\mid n\}=\delta(n)$. Hence
   \begin{align*}
    k!N_f(D, k, b)
  &\geq\frac{1}{n} {(m)_k}- (C_dnp^{-\frac 1d}+c+ (m- C_dnp^{-\frac 1d}-c)\delta(n)+k-1)_k\\
   &=\frac{1}{n} {(m)_k}- (\delta(n)m+(1-\delta(n)(C_dnp^{-\frac 1d}+c)+k-1)_k.
\end{align*}
 The second inequality follows from Lemma \ref{lem6.3}.

\section{The Finite Field Case $R=\f$}

 For our proof, we first recall Weil's character sum estimate in the
following form \cite{Wan}.
\begin{lem}[Weil]
Suppose $\psi$ is a non-trivial additive character of the additive group $\f$.
  Let $f(x)\in \f[x]$ be a polynomial of degree $d$ not divisible by $p$.
Then,
\begin{align*}
 |\sum_{x\in{\bf F}_q}\psi(f(x))|\leq(d-1)\sqrt{q}.
\end{align*}
\end{lem}

\begin{cor}\label{4.2}
Suppose $\psi$ is a non-trivial additive character of the additive group $\f$.
  Let $f(x)\in \f[x]$ be a polynomial of degree $d$ not divisible by $p$.
Suppose $D\subseteq\f$ and $|D|=q-c$. Then,
\begin{align}\label{5.1}
 |\sum_{x\in D}\psi(f(x))|\leq(d-1)\sqrt{q}+c.
\end{align}
\end{cor}
%
%
%
%
%

{\bf Proof of Theorem for $R=\f$}.
Write $|D|=m$. Applying Lemma \ref{lem3.1},  we have

 \begin{align*}
 k!N_f(D, k, b)&=\frac{1}{q} {(m)_k}+\frac{1}{q} \sum_{\psi\ne \psi_0}
 \psi^{-1}(b) \sum_{\tau\in
C_{k}}\sign(\tau)C(\tau) F_{\tau}(\psi),
  \end{align*}
  where  $C_{k}$ is the set of all conjugacy classes
 of $S_{k}$ and $C(\tau)$ counts the number of permutations conjugate to
 $\tau$, and
   \begin{align*}
F_{\tau}(\psi)=\prod_{i=1}^{k}(\sum_{a\in D}\psi^i(f(a)))^{c_i}.
\end{align*}
Applying equation \ref{5.1} in Corollary \ref{4.2}, we have 
$$|F_{\tau}(\psi)| \leq
m^{\sum_{i=1}^{k}c_im_i(\psi)}((d-1)q^{\frac 12}+c)^{\sum_{i=1}^{k}
c_i(1-m_i(\psi))},$$  where $m_i(\psi)$ is defined as follows:
 $ m_i(\psi)=1$ if $\psi^i=1$ and  $ m_i(\psi)=0$ if $\psi^i\neq1$.

Since the additive group of $\f$ is $p$-elementary,  for nontrivial
character $\psi$, $\order(\psi)=p$. Thus $\psi^i=1$ if and only if  $p\mid i$.
Assume $m \geq (d-1)q^{\frac 12}+c$. We deduce
   \begin{align*}
 k!N_f(D, k, b)&=\frac{1}{q} {(m)_k}+\frac{1}{q} \sum_{\psi, \order(\psi)=p}
 \psi^{-1}(b) \sum_{\tau\in
C_{k}}\sign(\tau)C(\tau) F_{\tau}(\psi)\\
 &\geq \frac{1}{q} {(m)_k}-\frac{q-1}{q}
 \sum_{\tau\in
C_{k}}C(\tau) m^{\sum_{j=1, p \mid
j}^kc_{j}}((d-1)q^{\frac 12}+c)^{\sum_{j=1, p\nmid j}^{k}
c_j}\\
 &\geq \frac{1}{q} {(m)_k}- (\frac{m}{p}+\frac {p-1}p ((d-1)q^{\frac 12}+c)+k-1)_k. \qedhere
\end{align*}
The last equality follows from Lemma \ref{lem6.4} in the case that $d=p$.

\section{The Case $f(x)=x$, $R=G$ and $D\subseteq G$ arbitrary}

{\bf Proof of Theorem for $R=G$}.  Let $\hat{G}$ be the group of
additive characters of $G$ and let $\psi_0$ be the principal
character sending each element in $G$ to 1.   Let $N(D, k, b)$ be the number of $k$-subsets
$S\subseteq D$ such that $\sum_{x\in S}x=b$.  Write $|G|=n$ and $|D|=m=n-c$.
Suppose $m>c$.
Applying Lemma \ref{lem3.1},  we have

 \begin{align*}
 k!N(D, k, b)&=\frac{1}{n} {(m)_k}+\frac{1}{n} \sum_{\psi\ne \psi_0}
 \psi^{-1}(b) \sum_{\tau\in
C_{k}}\sign(\tau)C(\tau) F_{\tau}(\psi),
  \end{align*}
  where  $C_{k}$ is the set of all conjugacy classes
 of $S_{k}$ and $C(\tau)$ counts the number of permutations conjugate to
 $\tau$, and
   \begin{align*}
F_{\tau}(\psi)=\prod_{i=1}^{k}(\sum_{a\in D}\psi^i(a))^{c_i}.
\end{align*}
A trivial character sum bound gives
$$|F_{\tau}(\psi)| \leq
m^{\sum_{i=1}^{k}c_im_i(\psi)}c^{\sum_{i=1}^k c_i(1-m_i(\psi))},$$
where $m_i(\psi)$ is defined as follows:
 $ m_i(\psi)=1$ if $\psi^i=1$ and  $ m_i(\psi)=0$ if $\psi^i\neq1$.

Let $e(G)$ be the exponent of $G$ and so it is also the exponent of
$\hat{G}$. Thus for nontrivial character $\psi$, $m_i(\psi)=0$ if $(e(G), i)=1$.  Since $m>c$, we have
   \begin{align*}
 k!N(D, k, b)&=\frac{1}{n} {(m)_k}+\frac{1}{n} \sum_{\psi\ne \psi_0}
 \psi^{-1}(b) \sum_{\tau\in
C_{k}}\sign(\tau)C(\tau) F_{\tau}(\psi)\\
&\geq \frac{1}{n} {(m)_k}- \frac{n-1}{n}
 \sum_{\tau\in
C_{k}}C(\tau) m^{\sum_{j=1, (e(G), j)>1}^kc_{j}}c
^{\sum_{j=1, (e(G), j)=1}^{k}
c_j}\\
 &\geq \frac{1}{n} {(m)_k}- (c + (m-c)\sum_{i\mid e(G), \mu(i)=-1}\frac {1}{i}+k-1)_k, \qedhere
\end{align*}
where the last equality follows directly from Lemma \ref{lem6.3}.

\section{The Case $f(x)=x$, $R=G$ and $D=G$}

{\bf Proof of Theorem for $D=R=G$}. The proof is quite similar as the last case.
In this case, $|D|=|G|=m$ and $c=0$.  Applying Lemma \ref{lem3.1},  we have

 \begin{align*}
 k!N(D, k, b)&=\frac{1}{n} {(n)_k}+\frac{1}{n} \sum_{\psi\ne \psi_0}
 \psi^{-1}(b) \sum_{\tau\in
C_{k}}\sign(\tau)C(\tau) F_{\tau}(\psi),
  \end{align*}
  where  $C_{k}$ is the set of all conjugacy classes
 of $S_{k}$ and $C(\tau)$ counts the number of permutations conjugate to
 $\tau$, and
   \begin{align*}
F_{\tau}(\psi)=\prod_{i=1}^{k}(\sum_{a\in G}\psi^i(a))^{c_i}.
\end{align*}
A trivial character sum computation gives
$$F_{\tau}(\psi)=
n^{\sum_{i=1}^{k}c_im_i(\psi)}0^{\sum_{i=1}^k c_i(1-m_i(\psi))},$$
where $m_i(\psi)$ is defined as follows:
 $ m_i(\psi)=1$ if $\psi^i=1$ and  $ m_i(\psi)=0$ if $\psi^i\neq1$.

Let $e(G)$ be the exponent of $G$ and so it is also the exponent of
$\hat{G}$. Thus for nontrivial character $\psi$, $m_i(\psi)=0$ if $(e(G), i)=1$.  We then have
   \begin{align*}
 k!N(D, k, b)&=\frac{1}{n} {(n)_k}+\frac{1}{n} \sum_{\psi\ne \psi_0}
 \psi^{-1}(b) \sum_{\tau=(c_1, c_2, \cdots, c_k)\in
C_{k}}\sign(\tau)C(\tau) F_{\tau}(\psi)\\
&=\frac{1}{n} {(n)_k}+\frac{1}{n} \sum_{1\ne d\mid n}\sum_{\psi, \order(\psi)=d}
 \psi^{-1}(b) \sum_{\tau \in C_k}\sign(\tau)C(\tau) n^{\sum_{i=1}^{k}c_im_i(\psi)}0^{\sum_{i=1}^k c_i(1-m_i(\psi))}.\\
 \end{align*}

 Since for $\tau=(c_1, c_2, \cdots, c_k)\in C_{k}, c_i=0 \Rightarrow m_i(\psi)=1$, we have

 \begin{align*}
  k!N(D, k, b)&=\frac{1}{n} {(n)_k}+\frac{1}{n} \sum_{1\ne d\mid n}\sum_{\psi, \order(\psi)=d}
 \psi^{-1}(b) \sum_{\tau \in C_k}\sign(\tau)C(\tau) n^{\sum_{i=1}^{k}c_i}\\
 &=\frac{1}{n} {(n)_k}+\frac{(-1)^k}{n} \sum_{1\ne d\mid n}\sum_{\psi, \order(\psi)=d}
 \psi^{-1}(b) \sum_{\tau \in C_k}(-1)^{c_1+c_2+\cdots+c_k}C(\tau) n^{\sum_{i=1}^{k}c_i}\\
  &=\frac{1}{n} {(n)_k}+\frac{(-1)^k}{n} \sum_{1\ne d\mid n}\sum_{\psi, \order(\psi)=d}
 \psi^{-1}(b) \sum_{\tau \in C_k}C(\tau) (-n)^{\sum_{i=1}^{k}c_i}\\
\end{align*}

From the formula given in Lemma \ref{lem6.4},
one has \begin{align*}
&\sum_{\tau=(c_1, \cdots, c_k)\in
C_{k} }
 C(\tau) (-n)^{\sum c_i}\\
  &=[\frac {t^k}{k!}]\frac {1}{(1-t^d)^{-n/d}}={-n/d+k/d-1 \choose k/d}=(-1)^{k/d}{n/d \choose k/d}.\\
\end{align*}

We then have
 \begin{align*}
  k!N(D, k, b) &=\frac{1}{n} {(n)_k}+\frac{(-1)^k}{n} \sum_{1\ne d\mid n, d\mid k}\sum_{\psi, \order(\psi)=d}
 \psi^{-1}(b) k!(-1)^{k/d}{n/d \choose k/d}\\
 &=\frac{1}{n} {(n)_k}+\frac{(-1)^k}{n} \sum_{1\ne d\mid (n,k)}k!(-1)^{k/d}{n/d \choose k/d}\sum_{\psi, \order(\psi)=d}
 \psi^{-1}(b).
\end{align*}
Thus
$$ N(D, k, b)=\frac{1}{n} {n \choose k}+\frac{1}{n} (-1)^{k+k/d}\sum_{1\ne d\mid (n, k)}{n/d \choose k/d}\sum_{\psi,
\order(\psi)=d} \psi(b),$$
 where  $\sum_{\psi, \order(\psi)=d}  \psi(b)$ is the Ramanujan sum.
  \qedhere

{ \bf Remark:}   This approach can be used to give explicit formulas when $G\backslash D$ is a very small constant.
%

%

{\bf Acknowledgements.} The authors wish to thank Professor Richard Stanley for his helpful suggestions.

\end{document}